\documentclass[12pt,a4paper]{article}
\usepackage{indentfirst}
\usepackage{amsmath,amssymb,amsfonts,amsthm,graphics}
\usepackage{makeidx}
\usepackage{amsmath,amssymb,amsfonts,amsthm,graphics,graphicx}
\usepackage{makeidx}
\usepackage{color}

\newtheorem{thm}{Theorem}[section]

\newtheorem{lem}{Lemma}[section]
\newtheorem{cor}{Corollary}[section]

\theoremstyle{definition}

\def\-{\mbox{--}}

\newtheorem{claim}{Claim}

\newtheorem{obser}{Observation}
\newtheorem{conj}{Conjecture}[section]

\begin{document}
\title{A Tur\'an-type problem on distance two\footnote{Supported by NSFC.}}
\author{\small  Xueliang Li, Jing Ma, Yongtang Shi, Jun Yue\\
\small Center for Combinatorics and LPMC-TJKLC\\
\small Nankai University, Tianjin 300071, China\\
\small Email: lxl@nankai.edu.cn, majingnk@gmail.com,
shi@nankai.edu.cn, yuejun06@126.com}
\date{}
\maketitle

\begin{abstract}
A new tur\'an-type problem on distances on graphs was introduced by Tyomkyn and
Uzzell. In this paper, we focus on the case that the distance is two. We primely
show that for any value of $n$, a graph on $n$ vertices without three vertices
pairwise at distance $2$, if it has a vertex $v \in V(G)$, whose neighbours are
covered by at most two cliques, then it has at most $(n^2 - 1)/4 + 1$ pairs of
vertices at distance $2$. This partially answers a guess of Tyomkyn and Uzzell
[Tyomkyn, M., Uzzell, A.J.: A new Tur\'an-Type promble on distaces of graphs. Graphs
Combin. {\bf 29}(6), 1927--1942 (2012)].


\end{abstract}

\section{Introduction}
In \cite{UT}, Tyomkyn and Uzzell introduced a new Tur\'an-type
problem on distances in graphs, which is an extension of the problem
studied by Bollob\'as and Tyomkyn in \cite{BT}, namely, determining
the maximum number of paths with length $k$ in a tree $T$ on $n$
vertices.

The problem on counting paths of a given length in a graph $G$ has
been studied since 1971, see, e.g., \cite{AK,A,BE,BS,BS1,By,REZ} and
the references therein. On the other hand, counting paths of length
$k$ in trees can be interpreted as counting pairs of vertices at
distance $k$. Tyomkyn and Uzzell asked a natural question as
follows.

{\bf Question.}  For a graph $G$ on $n$ vertices, what is the
maximum possible number of pairs of vertices at distance $k$?

Let $G=(V,E)$ be a connected simple graph. The distance between two
vertices $u$ and $v$ in $G$, denoted by $d_{G}(u,v)$, is the length
of a shortest path between $u$ and $v$ in $G$. Let $N_{G}(v)$ be the
neighborhood of $v$, and $d_{G}(v)=|N_{G}(v)|$ denote the degree of
vertex $v$. The greatest distance between any two vertices in $G$ is
the diameter of $G$, denoted by $diam(G)$. The set of neighbors of a
vertex $v$ in $G$ is denoted by $N(v)$ or $N^{1}(v)$, and the set of
vertices, whose distance is $i$ from $v$, is denoted by $N^{i}(v)$,
where $i \in \{1, 2, 3, \cdots, diam(G) \}$. Suppose that $V'$ is a
nonempty subset of $V$. The subgraph of $G$ whose vertex set is $V'$
and whose edge set is the set of those edges of $G$ that have both
ends in $V'$ is called the subgraph of $G$ induced by $V'$ and is
denoted by $G[V']$; we say that $G[V']$ is an induced subgraph of
$G$.  A clique in a graph $G$ is a subset of its
vertices such that every two vertices in the subset are connected by
an edge.

If $H$ is a graph, then we say that $G$ is $H$-free if $G$ does not
contain a copy of $H$ as an induced subgraph. The claw, denoted
by $C$, is the complete graph $K_{1,3}$ (see Figure \ref{fig1}). Thus, $G$
is said to be claw-free if it does not contain an induced subgraph
that is isomorphic to $C$.

\begin{figure}[h,t,b,p]
\begin{center}
\includegraphics[bb = 235 633 365 720, scale = 0.9]{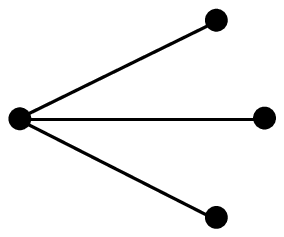}
\caption{A claw} \label{fig1}
\end{center}
\end{figure}

A graph $G$ is a \emph{quasi-line graph} if for every vertex $v \in V(G)$, the
neighborhood of $v$ can be partitioned into two sets $A$, $B$ in such a way that $A$
and $B$ are both cliques. (Note that there may be edges between $A$ and $B$.) Thus
all line graphs are quasi-line graphs, and all quasi-line graphs are claw-free, but
if we converse either of the statements, it is not true. For the other notations and
terminology we refer to \cite{BM}.

For a graph $G$, a new graph $G_k$ is defined to be the graph with
vertex set $V(G)$ and $\{x, y\}\in E(G_k)$ if and only if $x$ and
$y$ are at distance $k$ in $G$. We call $G_k$ the {\it distance-$k$
graph}. Such vertices $x$ and $y$ are called {\it $k$-neighbors}. We
call $d_{G_k}(x)$ the {\it $k$-degree} of $x$. Let $\omega(G)$
denote the clique number of graph $G$, which is the maximal number
of vertices of a clique in $G$.

It is interesting to maximize the number of edges in $G_k$ over all
graphs $G$ on $n$ vertices. In \cite{BT}, Bollob\'as and Tyomkyn
proved that if $G$ is a tree, then $e(G_k)$ is maximal when $G$ is a
$t$-broom for some $t$.
\begin{thm}
Let $n\geq k$. If $G$ is a tree on $n$ vertices, then $e(G_k)$ is
maximal when $G$ is a $t$-broom. If $k$ is odd, then $t = 2$. If $k$
is even, then $t$ is within $1$ of
$$\frac 1 4+\sqrt{\frac 1 {16}+\frac{n-1}{k-2}}.$$
\end{thm}
\noindent For a general graph $G$, Tyomkyn and Uzzell \cite{UT} gave
a conjecture and proved a part of it in the following theorem.

\begin{conj}
Let $k \geq 3$ and $ t \geq 2$. There is a function $h_2: N \times N
\rightarrow N$ such that if $n \geq h_2(k,t)$, then $e(G_k)$ is
maximised over all $G$ with $|G|=n$ and $\omega(G_k) \leq t$ when
$G$ is $k$-isomorphic to a $t$-broom for some $t$.
\end{conj}

\begin{thm}
There is a constant $k_0$ and a function $n_0: N\rightarrow N$ such
that for all $k \geq k_0$, all $n \geq n_0(k)$ and all graphs $G$ of
order $n$ with no three vertices pairwise at distance $k$,
\begin{center}
$e(G_k) \leq (n-k+1)^2/4$.
\end{center}
Moreover, if the equality holds, then $G$ is $k$-isomorphic to the
double broom.
\end{thm}

For the detailed proof and some terminology, we refer to \cite{UT}.
Actually, Tyomkyn and Uzzell try to do better about the bound, but
there is no good way. For the case of $k=2$, they believe that for
$n\geq5$, a triangle-free $G_2$ can have no more than  $(n-1)^2/4+1$
edges. They mentioned that this is clearly true for $n = 5$ and a
computer search verifies that it also holds for $6\leq n\leq 11$.
However, they cannot prove it in general. In this paper, we will
give an partially answer to it. Our main results are as follows:

\begin{thm}\label{thm1.1}
Let $G$ be a graph on $n$ vertices, which has no three vertices pairwise at distance
$2$. If there exists an vertex $v \in V(G)$, whose neighbors are covered by at most
two cliques, then it has at most $(n^2 - 1)/4 + 1$ pairs of vertices at distance
$2$.
\end{thm}

From theorem \ref{thm1.1}, we can get the following corollary.

\begin{cor}\label{cor1.1}
Let $G$ be a quasi-line graph on $n$ vertices, which has no three vertices pairwise
at distance $2$, then it has at most $(n^2 - 1)/4 + 1$ pairs of vertices at distance
$2$.
\end{cor}

\section{Preliminaries}
In this section, we will discuss some restrictions on the structure
of graph $G$ on condition that $G_2$ is triangle-free.

Let $C_k$ be the cycle with $k$ vertices. First of all, we
define two graphs $C^{'}_6$ and $C^{''}_6$ which can be obtained from $C_6$ as follows:
$C^{'}_6=C_6+v_1v_3$, $C^{''}_6=C_6+v_1v_3+v_3v_5$.

\begin{lem}\label{lem2.1}
If $G_2$ is triangle-free, then $G$ is claw-free and $C_6$-free, $C^{'}_6$-free, $C^{''}_6$-free.
\end{lem}

\begin{proof}
By contradiction. Suppose $G$ has a claw $C$ as an induced subgraph.
Let $V(C)=\{v, u_1, u_2, u_3\}$ and $v$ is adjacent to $u_i$ ($i=1,
2, 3$). Then the three vertices $u_1$, $u_2$ and $u_3$ form a
triangle in $G_2$, a contradiction.

Suppose that $G$ is not $C_6$-free. Let $C_6=v_1v_2\ldots v_6v_1$,
then $v_1v_3,\ v_3v_5,\ v_1v_5\in e(G_2)$, which implies that the
three vertices $v_1$, $v_3$ and $v_5$ form a triangle in $G_2$, a
contradiction.

Similarly, $G$ is $C^{'}_6$-free and $C^{''}_6$-free.
\end{proof}

The {\it stability number} $\alpha(G)$ of a graph $G$ is the
cardinality of the largest stable set. Recall that a stable set of
$G$ is a subset of vertices such that no two of them are
connected by an edge. For a claw-free graph, there is a well-known result
\cite{REZ} as follows.

\begin{lem}\label{lem2.2}
Let $G$ be a claw-free graph with stability number at least three, then every vertex $v$ satisfies exactly one of the following:

 (1) $N_G(v)$ is covered by two cliques;

 (2) $N_G(v)$ contains an induced $C_{5}$.

\end{lem}

From Lemma \ref{lem2.2}, we know that for a claw-free graph $G$, the subgraph induced by the neighborhood of a vertex $v\in G$ is covered by at most two cliques, or contains an included $C_{5}$. In the following, we give an observation about the subgraph induced by $N^2_G(v)$.

\begin{obser}\label{obser2.1}
Let $G$ be a graph with diameter two and $v\in V(G)$. If $G_2$ is triangle-free, then $G[N^2_G(v)]$ is a clique.
\end{obser}

Let $G^{(d)}$ denote a graph $G$ with diameter $d$. Let
$v_0v_1\ldots v_d$ be a spindle of $G^{(d)}$ and
$V_d=N_{G^{(d)}}(v_{d-1})\backslash v_{d-2}$. We define an operation
as follows: delete all the edges between $v_{d-1}$ and $V_d$, and
then join $v_{d-2}$ with all the vertices of $V_d$. We call this operation {\it ``move
$V_d$ to $v_{d-2}$"}.

\begin{lem}\label{lem2.3}
If $G_2$ is triangle-free, then $e(G^{(d)}_2)\leq e(G^{(d-1)}_2)$, where $G^{(d-1)}$ is obtained by applying the above operation on $G^{(d)}$.
\end{lem}

\begin{proof}
Let $v_0v_1\ldots v_d$ be a spindle of $G^{(d)}$. After applying the
above operation on $G^{(d)}$, the only change on the number of
vertex pairs $\{u, v\}$ such that $d_{G^{(d)}}(u, v)=2$ is brought by
the movement of $V_d$. Thus to prove $e(G^{(d)}_2)\leq e(G^{(d-1)}_2)$, it
suffices to show that for every spindle of $G^{(d)}$, the number of
vertex pairs such that the distance between them is two is not decreasing after
using the operation ``move $V_d$ to $v_{d-2}$".
So we only need to show that for some spindle, after using the above operation,
$|\{u|d_{G^{(d)}}(u, v_d)=2\}|\leq |\{u'|d_{G'}(u', v_d)=2\}|$,
where $G'$ is obtained by moving $V_d$ to $v_{d-2}$ in $G^{(d)}$.
Since $|\{u|d_{G^{(d)}}(u, v_d)=2\}|=|N_{G^{(d)}}(v_{d-1})\backslash
v_d|$ and $|\{u'|d_{G'}(u',
v_d)=2\}|=|N_{G^{(d)}}(v_{d-2})\backslash v_{d-2} \cup
N_{G^{(d)}}(v_{d-1})\backslash v_d\backslash v_{d-2}|$, we can get
that $|\{u|d_{G^{(d)}}(u, v_d)=2\}|\leq |\{u'|d_{G'}(u', v_d)=2\}|$.
Therefore, we have $e(G^{(d)}_2)\leq e(G^{(d-1)}_2)$.
\end{proof}

\begin{cor}\label{cor2.1}
If $G_2$ is triangle-free and $d$ ($d\geq 2$) is the diameter of
$G$, then $e(G^{(d)}_2)\leq e(G^{(2)}_2)$.
\end{cor}

\begin{proof}
By Lemma \ref{lem2.3}, we know that $e(G^{(d)}_2)\leq e(G^{(d-1)}_2)\leq \cdots \leq e(G^{(2)}_2)$.
\end{proof}

By Corollary \ref{cor2.1}, we know that to maximize $|e(G^{(d)}_2)|$
($d\geq2$), it suffices to get the maximum value of
$|e(G^{(2)}_2)|$. Thus, in the following part of the paper, we focus
on the graph $G^{(2)}$. For convenience, we write $G$ instead of
$G^{(2)}$.

\section{Proof of Theorem \ref{thm1.1}}
In this section, we give the proof of Theorem \ref{thm1.1}. For
convenience, we use $\{x,y\}$ or vertex-pair to stand for the vertex-pair such that $d_G(x,y)=2$. Actually, Theorem \ref{thm1.1} can be stated as the
following theorem.

\begin{thm}\label{thm3.1}
Let $G$ be a graph with $|V(G)| \geq 5$. If there is a vertex $v \in V(G)$ whose neighbourhood is covered by at most two cliques, then a triangle-free $G_2$ can have no more than $(n-1)^2/4+1$ edges.
\end{thm}

\begin{proof}

Since $G_2$ is triangle-free, Lemma \ref{lem2.1} implies that $G$ is
claw-free. Let $v\in V(G)$, whose neighbours is covered by at most two cliques. Then the proof will be given by the following cases.

{\bf Case 1.} $N_G(v)$ is covered by only one clique.

Let $V_1=N_G(v)$ and $V_2=N^2_G(v)$. By Observation \ref{obser2.1},
$G[N^2_G(v)]$ is a clique. Suppose $V_1=V_{11}\cup V_{12}$ and
$V_2=V_{21}\cup B\cup V_{22}$, where $V_{21}$ is only adjacent to
$V_{11}$, $V_{22}$ is only adjacent to $V_{12}$, $B$ is adjacent to
both $V_{11}$ and $V_{12}$. Without loss of generality, we suppose that
$|V_{11}|\geq |V_{12}|$ and $|V_{21}|\geq |V_{22}|$.

By considering whether $V_{12}=\emptyset$ or not, we give the discussion as follows.

{\bf Subcase 1.1.} $V_{12}=\emptyset$.

Since $V_{22}$ is only adjacent to $V_{12}$, we have
$V_{22}=\emptyset$. That is, $G[V_1,V_2]$ is a complete
bipartite graph. Hence, $e(G_2)=|V_2|\leq n-2\leq (n-1)^2/4+1$.

{\bf Subcase 1.2.} $V_{12}\neq \emptyset$.

In this subcase, we will give the proof in detail as follows.

{\bf Subsubcase 1.2.1.} $V_{22}\neq \emptyset$.

Let $d=|V_{22}|$. We can define a new graph $G'$ (see Figure \ref{fig2}) as follows. $V(G')=V(G)$, $V'_1=N_{G'}(v)= V'_{11} \cup V'_{12}$ and $V'_2=N^2_{G'}(v)= V'_{21} \cup V'_{22}$, where $G'[V'_1]$ and $G'[V'_2]$ are both cliques, $G'[V'_{11},V'_{21}]$ and $G'[V'_{12},V'_{22}]$ are both complete bipartite graphs, and $|V'_{11}|=1,~|V'_{12}|=|V_{11}|+|V_{12}|-1,~|V'_{21}|=|V_{21}|+|B|,~|V'_{22}|=|V_{22}|$.
\begin{figure}[h,t,b,p]
\begin{center}
\includegraphics[bb = 190 590 375 715, scale = 0.9]{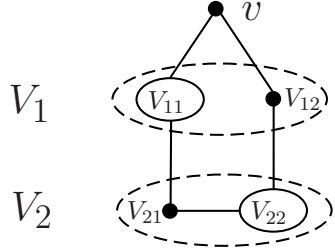}
\caption{A new graph $G'$} \label{fig2}
\end{center}
\end{figure}

Suppose that $d=1$. In the following, we show that $G'$ can be obtained from $G$ by applying the corresponding operations mentioned in the following paper. What is more, we show that such operations ensure that the number of vertices-pairs remains the same or increases.

Firstly, delete the edges between $V_{12}$ and $B$, and it is
obvious that the number of $\{u_1, u_2\}$ is not decreasing, where $u_1, u_2 \in V(G)$. Then $V'_{21}=V_{21}\cup B$ and $V'_{22} = V_{22}$.

Secondly, suppose $u \in V_{11}$ such that $u$ is adjacent to all the vertices in $V'_{21}$. Delete the edges between $V_{11}\setminus \{u\}$ and $V'_{21}$, meanwhile, connect all the vertices in $V_{11}\setminus u$ and all the vertices in $V'_{22}$. Then $V'_{11}=\{u\}$ and $V'_{12}= \{V_{11}\setminus \{u\}\} \cup V_{12}$. Therefore, we get the graph $G'$ (see Figure \ref{fig3}). Let $a=|V_{11}|,~b=|V_{12}|,~c=|V_{21}|+|B|$. Since $e(G_2)=a+bc+c+1$ and $e(G'_2)=(a+c-1)b+(a+c+1)$, then $e(G'_2)-e(G_2)=(a-1)b \geq 0$, that is, after applying the above operation, we ensure that the number
of $\{u, v\}$ is not decreasing.

For the graph $G'$, we have $e(G'_2) = xy + x +2$ where $x=|V'_{12}|,~y=|V'_{21}|$, such that $x+y=n-3$ and $x \geq 1,~ y \geq 1$. Thus, $e(G'_2)\leq (n-2)^2/4+2 < (n-1)^2/4 + 1$, where $n \geq 5$.

Now suppose that $d\geq 2$. Let $u_1 \in V_{21}$ and $u_2 \in V_{12}$. Delete the edges between $V_{21}\setminus \{u_1\} \cup B$ and $V_{22}$ and the edges between $V_{12} \setminus \{u_2\}$ and $V_{11}$; meanwhile move the vertices of $V_{21} \setminus \{u_1\} \cup B$ to $V_{11}$ (the new vertex set obtained is denoted by $V'_{12}$), and the vertices in $V_{12}\setminus \{u_2\}$ with $V_{22}$ (the new vertex set obtained is denoted by $V'_{21}$), therefore, we get the graph $G'$ (see Figure \ref{fig3}), where $V'_{11}=\{u_1\}$ and $V'_{22}=\{u_2\}$. Let $a=|V_{11}|,~b=|V_{12}|,~c=|V_{21}|+|B|$. Since $e(G_2)=ad+bc+c+d$ and $e(G'_2)=(a+c-1)(b+d-1)+(a+c-1+1)+1$, then $e(G'_2)-e(G_2)=a(b-1)+(c-1)(d-2)\geq 0$, that is, using the above operation we ensure that the number of $\{w_1, w_2\}$ is not decreasing, where $w_1,w_2 \in V(G')$.

For the graph $G'$, we have $e(G'_2) = xy + x +2$, where $x=|V'_{12}|,~y=|V'_{21}|$ such that $x+y=n-3$ and $x \geq 1,~ y \geq 1$. Thus, $e(G'_2)\leq (n-2)^2/4+2 < (n-1)^2/4 + 1$, where $n \geq 5$.

\begin{figure}[h,t,b,p]
\begin{center}
\includegraphics[bb = 152 615 440 715, scale = 0.9]{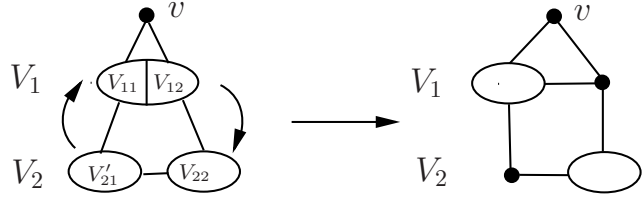}
\caption{A new graph $G'$} \label{fig3}
\end{center}
\end{figure}

{\bf Subsubcase 1.2.2.} $V_{22}=\emptyset$.

For this case, delete the edges between $V_{11}$ and $B$, then it returns to {\bf Subsubcase 1.2.1}.

By combining all the situations in {\bf Subcase 1.1}, we get that
$e(G)\leq (n-2)^2/ 4+ 2 < (n-1)^2/4 +1$, where $n \geq 5$.

{\bf Case 2.} $G[N_G(v)]$ is covered by two cliques.

In this case, $G[N_G(v)]$ is covered by two cliques, denoted by $V_1$ and $U_1$. According to the condition that whether $e(U_1,V_1) = 0$ or not, we prove it by the following subcases.

{\bf Subcase 2.1} $e(U_1,V_1) = 0$.

Let $V_2 = N^2_G(v)$. $V_2$ can be divided into three parts $A$, $B$ and $C$, where $A$ is only adjacent to $V_1$, $B$ is adjacent to both $V_1$ and $U_1$, and $C$ is only adjacent to $U_1$. Now we give the following claim.

\begin{claim}
$G[V_1, A]$ and $G[U_1, C]$ are both complete bipartite graphs.
\end{claim}

\begin{proof}
If $G[V_1, A]$ is not a complete bipartite graph, then there are
vertices $u_1 \in A$, $u_2 \in V_1$, $w \in U_1$ such that $d_G(u_1,u_2)=d_G(u_1,w)=d_G(u_2,w)=2$. Thus, the three vertices $u_1,u_2,w$ form a triangle in $G_2$, a contradiction.

Similarly, $G[U_1, C]$ is also a complete bipartite graph.
\end{proof}

Now we divide vertex set $B$ into three parts $B_1$, $B_2$ and $B_3$, where $B_1$ is only adjacent to $V_1$, $B_2$ is adjacent to both $V_1$ and $U_1$, and $B_3$ is only adjacent to $U_1$.

\begin{claim}
$G[V_1, B_1 \cup B_2]$ and $G[U_1, B_2 \cup B_3]$ are both complete bipartite graphs.
\end{claim}

\begin{proof}
To prove Claim 2, it suffices to show that every vertex $u \in
B$ is adjacent to all the vertices of $V_1$ or $U_1$ , that is, at least
one of $G[V_1, u]$ and $G[U_1, u]$ is a complete bipartite graph.
Suppose that there is a vertex $u \in B$ such that neither $G[V_1,
u]$ nor $G[U_1, u]$ is a complete bipartite graph. Then there are
vertices $w_1 \in V_1$, $w_2 \in U_1$ such that $d_G(u,w_1)=d_G(u,w_2)=d_G(w_1, w_2)=2$. Thus, $uw_1, w_1w_2, w_2u \in e(G_2)$, contradicting to the fact that $G_2$ is triangle-free.
\end{proof}

In the following, we apply some operations on $G$ to maximise $|e(G_2)|$.

We define a new graph $G'$ as follows. $G'$ with $V'_1=V_1$, $U'_1=U_1$, $V'_2=A' \cup C'$ ($A'=A \cup B_1 \cup B_2$ and $C'=B_3 \cup C$), where $G[V'_1, A']$, $G[V'_2, C']$ and $G[A', C']$ are complete bipartite graphs, $G[V'_1]$, $G[V'_2]$, $G[A']$ and $G[B']$ are complete graphs.

Form $G$ to $G'$, we perform the following operations: delete the edges between $B_1$ and $U_1$, remove the edges between $B_2$ and $U_1$, and the edges between $B_3$ and $V_1$. It is obviously that the number of vertices-pair at distance two is not decrease.

Now we construct a new graph $G''$ (see Figure \ref{fig4}) which is obtained from $G'$, that is, let $V(G'') = V(G')$, move the vertex in $A' \setminus \{u\}$ to $V'_1$, and the vertices in $C' \setminus \{w\}$ to $U'_1$. The new vertices sets are denoted by $V''_1$ and $U''_1$, where $u \in A'$ and $w \in B'$.

Let $a=|V'_1|$, $b=|V'_2|$, $c=|A'|$, $d=|B'|$. Since $e(G'_2)=ab+ad+bc+c+d$ and  $e(G''_2)=(a+c-1)(b+d-1)+(a+c-1)+(b+d-1)+2$, then $e(G''_2)-e(G'_2)=(c-1)(d-1)\geq 0$, that is, after using the move operation the number of vertex-pairs whose distance is two is not decreasing.

Now, for the graph $G''$, let $x=|V''_1|$ and $y=|U''_1|$. Then $e(G''_2)= xy+x+y+2$, where $x+y=n-3$ and $x,y \geq 1$. By some calculations, we get that $e(G''_2)\leq (n-1)^2/ 4+1$. And the equality holds if and only if  $n$ is odd and $x=y=(n-1)/2$, where $n \geq 5$.

\begin{figure}[h,t,b,p]
\begin{center}
\includegraphics[bb = 130 608 454 720, scale = 0.9]{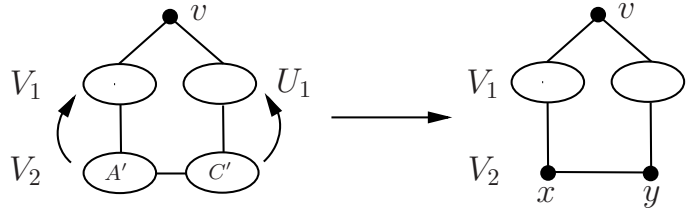}
\caption{A new graph $G''$} \label{fig4}
\end{center}
\end{figure}

Combining all the situations in {\bf Case 1}, $e(G_2)\leq (n-1)^2/
4+1$ follows on condition that $N_G(v)$ is covered by at least two
cliques.

{\bf Subcase 2.2} $e(U_1,V_1) \neq 0$.

In this subcase, if there is not a subset $B$ of $V_2$ such that the vertex in $B$
can form the vertex-pairs with some vertices of $V_{11} \subseteq V_1$ and meanwhile
with some vertices of $U_{11} \subseteq U_1$, then we delete the edges between $U_1$
and $V_1$. Now it returns to {\bf Subcase 2.1}.

If there is a subset $B$ of $V_2$ such that the vertices in $B$ can form the
vertex-pairs both with some vertices of $V_{11} \subseteq V_1$ and $U_{11} \subseteq
U_1$, then divide $B$ into two part $B_1$ and $B_2$, delete the edges between $V_1$
and $U_1$, move $B_1$ to $V_1$, and move $B_2$ to $U_1$. Now we want to prove that
there always exists such $B_1$ and $B_2$ to ensure the number of the vertex-pairs
not decreasing. Let $a=|V_{11}|$, $b=|U_{11}|$, $c=|B_1|$, $d=|B_2|$. By only
considering the vertex-pairs of those set, the change of the number is at least
$(a+c)(b+d) - (c+d)(a+b) = (a-d)(b-d) $, which is no less than $0$ ( by some
knowledge of the inequality, no matter how much $(a+c)$ and $(c+d)$ are, we can find
some number to make it right). Now this subcase returns to the {\bf Subcase 2.1}.

Combining all the cases, we complete the proof of Theorem \ref{thm3.1}.
\end{proof}

From the proof of Theorem \ref{thm3.1}, we can easily get Corollary \ref{cor1.1}. By
Lemmas \ref{lem2.1} and \ref{lem2.2}, if we can prove the following statement: all
the claw-free graphs with diameter two, which has no three vertices pairwise at
distance $2$, then it has at most $(n^2 - 1)/4 + 1$ pairs of vertices at distance
$2$, then we can confirm the guess of Tyomkyn and Uzzell.

\end{document}